\documentclass[11pt]{amsart}
\usepackage{amssymb,amsmath,amsthm}

\usepackage[  sfmathbb, partialup, nofligatures, ]{kpfonts}

\theoremstyle{theorem}
\newtheorem{theorem}{Theorem}
 \newtheorem{lemma}{Lemma}
 \newtheorem{corollary}{Corollary}
\theoremstyle{definition}

\begin{document}

\title{Sorry, the nilpotents are in the center}

\author{Vineeth Chintala}

\address{Indian Institute of Technology, Bombay, India}
\email{vineethreddy90@gmail.com}

\maketitle

\begin{abstract}
One of the key components of an algebra its set of nilpotents.  We give a simple proof of a classic result which states that a finite ring is commutative if all its nilpotents lie in the center.
\end{abstract}
\textbf{Keywords} : Noncommutative rings, idempotent, nilpotent, finite rings
\section*{Introduction}
\indent 

At the beginning of the 20th century, Wedderburn showed that 
\emph{every finite division ring is a field}. This startling result, along with the structure theory of rings, lead to many interesting results in 20th century Algebra, by Jacobson, Herstein, Kaplansky and others. Perhaps the most folklore of these is the result of Jacobson \cite{J}, which says that if all the elements of a ring satisfy the equation $x^{n(x)} = x$, then the ring is commutative. Jacobson's result has been further generalized, in the same direction, by Herstein and other mathematicians, showing that a ring is commutative if the elements satisfy certain general polynomial identities.

Due to the restriction of multiplicative freedom in the finite case, one may expect an elegant generalization of Wedderburn's result for finite rings.
Indeed, Herstein observed in \cite{H1} that :\\
\vskip 1mm
\noindent \emph{ If every nilpotent of a finite ring lies in its center, 
then the ring is commutative.}\\
\vskip 1mm
 As one can see, this generalizes Wedderburn's theorem (as division rings have no nilpotent elements). 
Unfortunately, this beautiful result isn't mentioned in graduate textbooks and is far from being well-known. Part of the reason may be due to the fact that it was proved as a corollary to more general (and technical) results in ring theory. In this paper, we give a simple proof of Herstein's result and hope to bring it to a wide audience. The proof is especially instructive to students as an illustration of how a result about the structure of an algebra can be strung together by figuring out the interplay of its idempotents and nilpotents. Indeed, understanding that nilpotents play a key role led to the notion of a \emph{radical} which was critical for the structure theory of algebras \cite{H2, L}. We will not get into this; Instead we hope to illustrate the power of nilpotents and idempotents by presenting an elementary result.
The ideas that lead to the proof are not original by any means and can be mined by studying the classic commutativity results of the 20th century.

A reader looking for a taste of noncommutative ring theory can find many excellent sources, including \cite{H2, L, R}. The book by Schafer \cite{S} is a nice introduction to the nonassociative world of Lie and Jordan algebras. Since Wedderburn's result is well-known, we will not prove it here again; (An elementary proof can be found in \cite{AZ}).

\subsection{Notation}
An element $x$ is called an idempotent if $x^2=x$ and a nilpotent if $x^n =0$. 
The center of the ring is denoted by $C$ and its set of nilpotents is denoted by $N$.

\section{Finite rings with central nilpotents. }
\noindent We will prove the main result by combining three presumably well-known lemmas. 

\vskip 2mm
\begin{lemma}\label{lemma1}
Let $R$ be a ring and $e$ an idempotent. If $e$ commutes with all the nilpotents of $R$, then $e$ is central. 
\end{lemma}

\begin{proof}
 For any $x\in R$, we have $(xe - exe)^2 = 0$. As $e$ commutes with all nilpotents,
\[xe -exe = (xe - exe)e = e(xe - exe) = 0.\] Similarly  $ex = exe$, thus $e\in C$. 
\end{proof}

\vskip 5mm

\begin{lemma}\label{lemma2}
Let $R$ be a finite ring with no non-trivial idempotents. Let $x \in R$. Then either $x$ is a nilpotent or $x$ is invertible
\end{lemma}

\begin{proof}
The set $\{x^{2^i} :  i \geq 1\}$ is finite and so $x^{2^r} = (x^{2^r})^t$ for some $r,t > 1$. 
\vskip 1mm
Let $y = x^{2^r}$. Then $y^{t-1}$ is an idempotent. Indeed, 
\[ y^{2(t-1)} = y^{t-2}y^t = y^{t-2}y = y^{t-1}.\]
Since $R$ has no nontrivial idempotents, we have $y^{t-1} \in \{0,1\}$.
\end{proof}

\vskip 5mm

\begin{lemma}\label{lemma3}
Suppose $[R,R] \subseteq C$. 
Let $b\in R$ such that $p^nb =0$. Then $b^{p^n} \in C$. 
\end{lemma}
\begin{proof}
By our hypothesis, the element $c = ab -ba$ commutes with $b$. Multiplying on the right by $b$, we get 
\[ab^2 =bab +cb = b^2a +2cb.\]
Continuing this way, we get $ab^i =b^ia + icb^{i-1}$. Thus $ab^{p^n} =b^{p^n}a$, since $p^nb = 0$.
\end{proof}

\vskip 5mm

\begin{theorem}
If all the nilpotents of a finite ring $R$ lie in its center, then the ring is commutative.
\end{theorem}

\begin{proof}
If $R = R_1 \oplus R_2$, then the theorem holds for $R$ if it holds for each $R_i$. So it is enough to consider the case where $R$ cannot be decomposed further.
Then,
\begin{itemize}
\item \emph{$R$ has no non-trivial idempotents} : If $e$ is an idempotent then $e$ is central by Lemma ~\ref{lemma1}, and so  $R \simeq Re\oplus R(1-e)$.
\vskip 1mm
\item \emph{$p^n \cdot R =0$ for some prime $p$} : If $m_1m_2\cdot R =0$ with $(m_1,m_2) = 1$, then $R \simeq R_1 \oplus R_2$ where $m_i\cdot R_i =0$. 
\end{itemize}

By our hypothesis,  $N \subseteq C$. In particular, $N$ is an ideal. By Lemma ~\ref{lemma2}, any $x \notin N$ is invertible. Then $\bar{x} \in R/N$ is also invertible and so (by Wedderburn's result) $R/N$ is a finite field with $p \cdot R/N = 0$. In particular $R/N$ is commutative, hence \[ [R,R] \subseteq N \subseteq C.\]

 In addition, the set $\{b^{p^{in}} :  i \geq 1\}$ is finite and so $b^{p^{rn}} = b^{p^{sn}}$ for some $r<s$. Therefore
$(\bar{b}^{p^{tn}} - \bar{b})^{p^{rn}} =0 $ {\small (where $t =s-r$)}, thus \[b^{p^{tn}} -b \in N.\]  
Since $b^{p^{tn}} \in C$ (by Lemma ~\ref{lemma3}), it follows that $b\in C$ for all $b \in R$.
 \footnote{It has come to the author's notice that another proof has been presented by Steve Ligh \cite{SL}. Unfortunately, that article is no longer accessible (online or offline) and the author hasn't been able to confirm if the proof is similar. }
\end{proof}
\vskip 5mm

As a corollary, we get the finite-version of Jacobson's theorem (Theorem 12.10, \cite{J}). 

\begin{corollary}
Let $R$ be a finite ring. If every $x \in R$ satisfies an equation of the type $x^{n(x)} =x$ for some $n(x) >1$, then $R$ is commutative.
\end{corollary}

\begin{proof}
Since  $x^{n(x)} =x$, the zero element is the only nilpotent and we can simply use the above theorem. 
\end{proof}

\vskip 5mm

\begin{theorem}
Consider the matrix ring $M_n(R)$ where $n >1$ and $R$ is any associative ring with unity.
Then any element of $M_n(R)$ which commutes with all the nilpotents lies in the 
center. 
\end{theorem}

\begin{proof}
The proof is a straight forward
computation; we illustrate it for $3\times 3$ matrices with entries from the ring $R$.
We will reduce the problem to $2\times 2$ matrices with entries from $R$.
Let
$A = (a_{ij})$ be a matrix which commutes with all the nilpotents.\\
Let $ B = \left( \begin{array}{ccc}
0 & 0 & x \\
0 & 0 & 0 \\
0 & 0 & 0 \end{array} \right)$.\\

 Now $AB = \left( \begin{array}{ccc} 
0 & 0 & a_{11}x \\
0 & 0 & a_{21}x\\
0 & 0 & a_{31}x \end{array} \right)$ 
$= BA = \left( \begin{array}{ccc} 
xa_{31} &  xa_{32} & xa_{33} \\
0       &    0     & 0\\
0       &    0     & 0\end{array} \right)$.\\

By substituting $x=1$, we get $a_{31} = a_{32} = 0.$
Moreover we have $a_{11}x= xa_{33}$ for any $x\in R$. Therefore $a_{11} = a_{33}$ lies in the center of $R$.
Since the transpose of $A$ commutes with all the nilpotents,
we similarly conclude that $a_{13} = a_{23} = 0$. \\

Hence $A = \left( \begin{array}{ccc} 
A'   & 0\\
0    & c\end{array} \right)$,
where $A' =  \left( \begin{array}{ccc} 
c      & a_{12} \\
a_{21} & a_{22}\end{array} \right)$ and $c$ lies in the center of $R$.
We also require that $A'$ commutes with all the nilpotents in the $2\times 2$ matrix ring with entries from $R$.
Since both $A'$ and its transpose commute with the nilpotent matrix $\left( \begin{array}{ccc} 
0   & 1\\
0   & 0\end{array} \right)$, 
it follows that $A'= \left( \begin{array}{ccc} 
c   & 0\\
0   & c\end{array} \right)$. Hence $A$ is a diagonal matrix with entries 
from the center of $R$.
\end{proof}
\vskip 5mm

\textbf{Acknowledgements}: I would like to thank the Indian Institute of Technology, Bombay for funding my research.

\end{document}